\pdfoutput=1
\documentclass[a4paper]{amsart}

\usepackage{bm}

\usepackage{graphicx,epsfig,color} 
\usepackage{multicol}
\usepackage{bm}
\usepackage{amsfonts}
\usepackage{rotating}
\usepackage[all]{xy}
\usepackage{amsmath,amscd,amsthm}
\usepackage{boxedminipage}
\usepackage{amssymb}

\newcommand{\R}{\mathbb{R}} 

\newtheorem{definition}{Definition}
 \newtheorem{lemma}{Lemma}

\newtheorem{proposition}{Proposition}
\newtheorem{exercise}{Exercise}
\newtheorem{problem}{Problem}

\newtheorem{theorem}{Theorem}

\newcommand{\df}{\stackrel{\mathrm{def}}{=}} 
\newcommand{\REF}[1]{{\normalfont (\ref{#1})}} 
\newcommand{\virg}[1]{``#1''} 
\renewcommand{\qed}{} 

\newcounter{UNI}

\newtheorem{ex}{Exercise}[UNI]

\newcounter{UNIp}

\newtheorem{pr}{Problem}[UNIp]

\begin{document}

\title{On the canonical connection for smooth envelopes}
\author{Giovanni Moreno}
\address{Mathematical Institute in Opava,\\
Silesian University in Opava,\\
Na Rybnicku 626/1, 746 01 Opava, Czech Republic.}
\email{Giovanni.Moreno@math.slu.cz}

\maketitle

 \begin{abstract}
A notion known as smooth envelope, or superposition closure, appears naturally in several approaches to generalized smooth manifolds which were proposed in the last decades. Such an operation is indispensable in order to perform differential calculus. A derivation of the enveloping algebra can be restricted to the original one, but it is a delicate question if the the vice--versa can be done as well. In a physical language, this would correspond to the existence of a canonical connection. In this paper we show an example of an algebra which always possesses such a connection.
\end{abstract}

\tableofcontents
\section*{Introduction}

The process of generalizing   differential calculus to commutative algebras---and even more general objects---began almost one century ago (probably with the work of K\"ahler in the thirties, see \cite{MR0238823}), and, besides aesthetic achievements, it introduced a lot of powerful tools for mathematical and theoretical  physics (see, e.g., \cite{Sarda, Nestruev}). The Polish school (see, e.g., \cite{MR1156977,MR2140786} and references therein) initiated by Sikorski \cite{MR0482794} in the seventies led to the notions of \emph{differential space}, and \emph{superposition closure}, which parallels the \emph{smooth envelope} of the later Russian school (see \cite{Nestruev}, \textsection 3.38). Also the   work of Michor \cite{MR583436} concerning manifolds of mappings can be framed in this development. 

The classical theory of smooth manifolds can nowadays be seen as a sub--theory of commutative algebra, since any smooth manifold is the spectrum of a suitable algebra (henceforth called \emph{smooth}, following \cite{Nestruev}). This observation is rather old, and can be traced back to Gel'fand and Kolmogorov (1939, see \cite{MR0407579}), though some call it \virg{the Milnor\&Stasheff exercise} (see \cite{MR0440554}) or \virg{Spectral Theorem} (see \cite{Nestruev}). Undoubtedly, a significant impulse in this direction was given by the parallel work of Groethendieck in algebraic geometric context (see \cite{MR0173675}, Chapter 20, \emph{D\'erivations at diff\'erentielles}) during the sixties. Similarly, the theory of vector bundle has become a part of the theory of projective modules, thanks to the celebrated Swan--Serre theorem \cite{MR0143225} (see also \cite{Nestruev}, Chapter 12, for an alternative proof). The reader may find useful references in \cite{MR1289393}.

This paper stems from the following elementary observation:  the spectrum of the tensor product of two algebras over the same ground field is the cartesian product of their spectra. But smooth function algebra on the cartesian product of two smooth manifolds, say, $M$ and $N$, does not coincide with the tensor product (over $\R$) of their respective smooth functions algebras. Indeed, $C^\infty(M)\otimes_\R C^\infty(N)$ is a proper subalgebra of $C^\infty(M\times N)$ and the latter is the  {smooth envelope} of the former. The main result (Theorem \ref{thExtSmEnv}) states that the inclusion  $C^\infty(M)\otimes_\R C^\infty(N)\subset C^\infty(M\times N)$ is equipped with a canonical connection. \par
As a technical preparation, we propose an obvious module--theoretic generalization of the  notion of   smooth envelope, thus   called   \emph{smoothening}  procedure (\ref{subsubAllisciamentoDiModuli}),  and show how differential forms behave under such operation.\par
Throughout this paper, $A$ is an algebra such that $\overline{A}$ is  smooth. In other words, the common spectrum of $A$ and $\overline{A}$ is a smooth manifold  $|A|$, whose smooth function algebra is $\overline{A}$.  In particular, this means that $A$ is constituted of functions on its spectrum, i.e., it is geometric.

\section{Preparatory results}
%
%
\subsection{Smooth envelopes}
Elements of $\overline{A}$ are of the form $H\circ \bm{a}$, where $\bm{a}\in A^k$ is a $k$--tuple of elements of $A$, understood as functions on $|A|$, and $H\in C^\infty(\R^k)$. The $\R$--algebra structure in $\overline{A}$ is given by:
\begin{eqnarray*}
r\cdot(H\circ \bm{a})&=&(rH)\circ \bm{a},\\
H\circ \bm{a}+H^\prime\circ \bm{a}^\prime&=&(H+H')\circ (\bm{a}\oplus\bm{a'}),\\
(H\circ \bm{a})\cdot(H'\circ \bm{a'})&=&(H\cdot H')\circ (\bm{a}\oplus\bm{a'}),
\end{eqnarray*}
where $r\in\R$,  $H\in C^\infty(\R^k)$, $H'\in C^\infty(\R^{k'})$, $\bm{a}\in A^k$, $\bm{a'}\in A^{k'}$, and $\bm{a}\oplus\bm{a'}$ belongs to $A^{k}\oplus A^{k'}=A^{k+k'}$. Functions $H+H'$ and $H\cdot H'$ are understood to be smooth functions on   $\R^{k+k'}$, due to natural embeddings of $C^\infty(\R^k)$ and $C^\infty(\R^{k'})$ into $C^\infty(\R^{k+k'})$.
 \subsection{Smoothening of modules}\label{subsubAllisciamentoDiModuli}
If $P$ is an $A$--module, then $P$ can be understood as the module of sections of the pseudobundle   $|P|\longrightarrow|A|$ (we use the same terminology as \cite{Nestruev}). Therefore, it seems natural to extend the scalar multiplication of sections to the elements of   $\overline{A}$. 
\begin{definition}\label{defAllisciamentoDiModuli}
The $\overline{A}$--module 
\begin{equation}
\overline{P}\df\overline{A}\otimes_AP
\end{equation}
is called the \emph{smoothening} of the $A$--module $P$.
\end{definition}
Thus, the smoothening of $P$   only results in an enlargement of the algebra of  scalars,  but does not affect the geometry  of the corresponding bundle, as  Lemma \ref{LemSmothLocFree} below shows. 
\begin{lemma}\label{LemSmothLocFree}
If an $A$--module $P$ is locally free over a cover $\mathcal{U}$ of $|A|$, the same holds for its smoothening $\overline{P}$. Moreover, local dimension is preserved.
\end{lemma}
\begin{proof}
Let $S_U$ (resp. $\overline{S}_U$) the multiplicative subset of $A$ (resp. $\overline{A}$) determined by $U\in\mathcal{U}$ (for the notation we follow \cite{Nestruev}). It is easy to verify that $\overline{S}_U^{-1}\overline{P}$ is isomorphic to $(\overline{S}_U^{-1}\overline{A})\otimes_{S_U^{-1}A}S_U^{-1}P$. Indeed maps
\begin{eqnarray}
\overline{S}_U^{-1}\overline{P} & \longrightarrow & (\overline{S}_U^{-1}\overline{A})\otimes_{S_U^{-1}A}S_U^{-1}P\\
\frac{f\otimes p}{g} & \longmapsto & \frac{f}{g}\otimes \frac{p}{1}\nonumber
\end{eqnarray}
and
\begin{eqnarray}
(\overline{S}_U^{-1}\overline{A})\otimes_{S_U^{-1}A}S_U^{-1}P &\longrightarrow &\overline{S}_U^{-1}\overline{P}   \\
\frac{f}{g}\otimes\frac{p}{h} & \longmapsto & \frac{f\otimes p}{gh}\nonumber
\end{eqnarray}
are well defined and inverse one to another.\par
Therefore, if $S_U^{-1}P$ is a free $S_U^{-1}A$--module, then $\overline{S}_U^{-1}\overline{P}$ is also free as $\overline{S}_U^{-1}\overline{A}$--module. Since this is true for any $U\in\mathcal{U}$, the result is proved.\qed
\end{proof}
\subsection{Smoothened tensor product}
Let  $P$ be a $C^\infty(M)$--module, and $Q$ a $C^\infty(N)$--module.   Next definition generalizes    the tensor product $\otimes_\R$, so that the multiplication of $P$ and $Q$ results into a $C^\infty(M\times N)$--module, rather than a $C^\infty(M)\otimes_\R C^\infty(N)$--module.
\begin{definition} The $C^\infty(M\times N)$--module
\begin{equation}
P\overline{\otimes}_\R Q\df\overline{P\otimes_\R Q}
\end{equation}
is called the \emph{smoothened tensor product} of $P$ and $Q$.
\end{definition}
\begin{proposition}\label{propSmoothenedTenProd}
The following isomorphisms holds:
\begin{eqnarray}
C^\infty(M)\overline{\otimes}_\R Q&=&C^\infty(M\times N)\otimes_{C^\infty(N)} Q\\
P\overline{\otimes}_\R C^\infty(N)&=&P\otimes_{C^\infty(M)}C^\infty(M\times N)
\end{eqnarray}
\end{proposition}
\begin{proof}
Just using the definition of smoothened tensor product, one sees that the assignments $H_{ij}\otimes(f^i\otimes q^j)\longmapsto (H_{ij}f^i)\otimes q^j$ and $H_i\otimes q^i\mapsto H_i\otimes(1_{C^\infty(M)}\otimes q^i)$, from $C^\infty(M\times N)\otimes_{C^\infty(M)\otimes_\R C^\infty(N)}\left(C^\infty(M)\otimes_\R Q\right)$ to $C^\infty(M\times N)\otimes_{C^\infty(N)} Q$ and viceversa,  are well--defined module homomorphisms inverse one of the other. Similarly for the second relation.\qed
\end{proof} 
\section{Canonical connection in the smooth envelope}
\subsection{Derivations along a  smooth envelope}
Let $P$ be an $\overline{A}$--module. The algebra morphism $\iota:A\subseteq\overline{A}$ allows to regard $P$ as an $A$--module $P_\iota$.
\begin{definition}
 An element  of the $A$--module $D(A,P_\iota)$  is a \emph{derivation of $A$ along its smooth envelope}.
\end{definition}
Since  $D(A,P_\iota)$ is also an $\overline{A}$--module with operation $(\overline{a}\cdot X)(a)\df \overline{a}X(a)$,   $\overline{a}\in\overline{A}$, $X\in D(A,P_\iota)$, $a\in A$, the correspondence $P \longmapsto D(A,P_\iota)$ is   an endofunctor in the category of $\overline{A}$--modules.
\begin{proposition}
 The restriction of   derivations of $\overline{A}$   to the subalgebra $A$ defines a surjective natural transformation
 \begin{equation}\label{EqProjFunct}
\Pi: D(\overline{A},\cdot\,) \longrightarrow D(A,\,(\cdot)_\iota\,)
\end{equation}
of endofunctors in the category of $\overline{A}$--modules.
\end{proposition}
Notice that the representative object of   functor $D(A,\,(\cdot)_\iota\,)$ is precisely the smooth\-ening $\overline{\Lambda^1(A)}=\overline{A}\otimes_A\Lambda^1(A)$ of the $A$--module $\Lambda^1(A)$.
\begin{lemma}
 Natural transformation $\Pi$ defined by \REF{EqProjFunct} is dual to the   $\overline{A}$--module homomorphism
\begin{equation}\label{EqProjFunctDual}
\varphi: \overline{\Lambda^1(A)}\longrightarrow \Lambda^1(\overline{A}),
\end{equation}
$\varphi(\overline{a}\otimes da) \df \overline{a}da$, $\overline{a} \in \overline{A}$, $a \in A \subset \overline{A}$.
\end{lemma}
Since $|\iota|$ is, in fact, the identity of $|A|$, one may expect that \emph{all} derivations of $\overline{A}$ are derivations along the smooth envelope, i.e., that $\varphi$ is a natural isomorphism.
 However, in general, this is not the case.\par
Observe that invertibility of $\varphi$ means the presence of a connection in the smooth envelope, i.e., a natural right inverse
 \begin{equation}\label{EqProjFunctINV}
D(A,\,(\cdot)_\iota\,)\stackrel{\nabla}{\longrightarrow } D(\overline{A},\cdot\,)  
\end{equation}
of $\Pi$, which allow to lift a derivation $X$ of $A$ to a derivation $\nabla_X$ of $\overline{A}$ (see, for instance, the book \cite{DeParis}, or \cite{Kra1998}).
\subsection{1--forms on smoothened product algebras}
Let $M$ and $N$ be smooth manifolds.
\begin{proposition}
If $A=C^\infty(M)\otimes_\R C^\infty(N)$, then $\overline{\Lambda^1(A)}$ is the module of sections of a smooth vector bundle over $M\times N$ of dimension $\dim(M)\dim(N)$.
\end{proposition}
\begin{proof}
 Straightforward. \qed
\end{proof}
\begin{theorem}\label{thExtSmEnv}
If $A=C^\infty(M)\otimes_\R C^\infty(N)$, then the homomorphism $\varphi$ defined by \REF{EqProjFunctDual} is bijective. In particular, $\Lambda^1(\overline{A})\cong\overline{\Lambda^1(A)}$.
\end{theorem}
\begin{proof}
Since  $A\subset\overline{A}=C^\infty(M\times N)$, any element $H\circ\bm{a}\in \overline{A}$ is the composition of the two smooth maps $\bm{a}:M\times N\to\R^k$ and $H:\R^k\to\R$. Therefore, by taking the differential of $H\circ\bm{a}$, one gets
\begin{equation}\label{eqChRul1}
d(H\circ\bm{a})=\left(\frac{\partial H}{\partial t^i}\circ\bm{a}\right)da^i.
\end{equation}
Recall  that $\Lambda^1(\overline{A})$ is generated by its subset $d\overline{A}$.    However, 
relation \REF{eqChRul1} allows to reduce   such a set of generators, by replacing it  by   the smaller subset $dA$. 
 Hence,    $\varphi$  is surjective.\par
Both $\Lambda^1(\overline{A})$ and $\overline{\Lambda^1(A)}$   are the modules of sections of a  $\dim(M)\dim(N)$--dimen\-sional vector bundle over $M\times N$, so $\varphi$ must also be injective.\qed
\end{proof}
 Theorem \ref{thExtSmEnv}  contains  a   global statement, whose  local (i.e., fiber--wise) analog is perhaps more intuitive. Namely, if one is interested in   the cotangent space $T_h\overline{A}= {\mu_h}/{\mu_h^2}$ (where $\mu_h\subseteq\overline{A}$), then  it is sufficient to work with the submodule $\Lambda^1(A) \subseteq \Lambda^1(\overline{A})$, instead of the whole   $ \Lambda^1(\overline{A})$.

\subsubsection*{Acknowledgements}
The author is thankful to the   Grant Agency  of the Czech Re\-public (GA \v CR)
for financial support under  the project P201/12/G028.

\end{document}